\documentclass[a4paper, 11pt]{article}

\usepackage[latin1]{inputenc}
\usepackage{subfigure}
\usepackage{tikz}
\usepackage{multicol}
\usepackage{amsmath}
\usepackage{amssymb}
\usepackage{amsthm}
\usepackage{fullpage}

\newtheorem{thm}{Theorem}[section]
\newtheorem{lem}[thm]{Lemma}

\newtheorem{prop}[thm]{Proposition}
\newtheorem{defi}[thm]{Definition}
\newtheorem{conj}[thm]{Conjecture}

\author{Wenjie Fang\footnote{This work is partially supported by ANR IComb (ANR-08-JCJC-0011) and ANR Cartaplus (ANR-12-JS02-001-01).} \\ LIAFA, Universit\'e Paris Diderot - Paris 7 \\ B\^atiment Sophie Germain, 75205 Paris Cedex 13, France \\ Email: Wenjie.Fang@liafa.univ-paris-diderot.fr}
\title{Bijective proofs of character evaluations using trace forest of the jeu de taquin}

\begin{document}
\maketitle
\begin{abstract}
Irreducible characters in the symmetric group are of special interest in combinatorics. They can be expressed either combinatorially with ribbon tableaux, or algebraically with contents. In this paper, these two expressions are related in a combinatorial way. We first introduce a fine structure in the famous jeu de taquin called ``trace forest'', with which we are able to count certain types of ribbon tableaux, leading to a simple bijective proof of a character evaluation formula in terms of contents that dates back to Frobenius (1901). Inspired by this proof, we give an inductive scheme that gives combinatorial proofs to more complicated formulae for characters in terms of contents.
\end{abstract}

\newcommand{\mysqr}[1]{rectangle +(0.5,0.5) +(0.25,0.25) node{#1}}
\newcommand{\mycpr}[1]{<^{(#1)}}
\newcommand{\mycpa}[1]{\vee^{(#1)}}
\newcommand{\sgn}{\operatorname{sgn}}
\newcommand{\mydeg}{\operatorname{deg}}

\section{Introduction}

Irreducible characters in the symmetric group has long attracted attention from combinatorists and group theorists. When evaluated at particular partitions, they can be expressed in terms of contents. Their study dates back to Frobenius. In \cite{frobenius1900charaktere}, for a partition $\lambda$ of an integer $n$, the following evaluations were given:

\begin{align*}
n (n-1) \chi^{\lambda}_{2,1^{n-2}} &= 2 f^{\lambda} \left( \sum_{w \in \lambda} c(w) \right), \\
n (n-1) (n-2) \chi^{\lambda}_{3,1^{n-3}} &= 3 f^{\lambda} \left( \sum_{w \in \lambda} (c(w))^2 + n (n-1) / 2 \right), \\
n (n-1) (n-2) (n-3) \chi^{\lambda}_{4,1^{n-4}} &= 4 f^{\lambda} \left( \sum_{w \in \lambda} (c(w))^3 + (2n-3) \sum_{w \in \lambda} c(w) \right).
\end{align*}

Here, $\chi^{\lambda}_{\mu}$ is the irreducible character indexed by $\lambda$ evaluated on the conjugacy class indexed by another partition $\mu$ of $n$, $f^{\lambda}$ is the dimension of the corresponding representation, and we sum over cells $w$ in the Ferrers diagram of $\lambda$, $c(w)$ is called the content of $w$. We postpone detailed definitions for these notions and related ones to Section~\ref{sec:prelim}. We observe that these character evaluations can be expressed with sums over powers of contents called content evaluations. This fact was proved in \cite{corteel2004content} for the general case, and in \cite{lassalle2008explicit} an explicit formula was given for general $\mu$ in $\chi^{\lambda}_{\mu}$. 

Such character evaluation in terms of contents are mostly obtained in an algebraic way, either using the Jucys-Murphy elements (\textit{e.g.} \cite{diaconis1989applications}), or with the help of symmetric functions (\textit{e.g.} \cite{corteel2004content, lassalle2008explicit}). They are also related to shifted symmetric functions on parts of partition (\textit{e.g.} \cite{kerov1994polynomial}). On the other hand, there is a well-developed combinatorial representation theory of the symmetric group (\textit{c.f.} \cite{stanley2001enumerative, sagan2001symmetric}), in which we can express characters combinatorially in terms of ribbon tableaux. It is thus interesting to relate ribbon tableaux to content evaluations using combinatorial tools, for example Sch\"utzenberger's famous jeu de taquin. Furthermore, since functions on contents appear in many contexts, such as in the proof that the generating function of some family of combinatorial maps is a solution to the KP hierarchy (\textit{c.f.} \cite{goulden2008kp}), a better understanding of the combinatorial importance of contents would also help us to better understand other combinatorial phenomena related to contents.

In this article, we look into the fine structure in the jeu de taquin. In Section~\ref{sec:jdt}, we define a notion called ``trace forest'' for skew tableaux that encapsulates the paths of all possible jeu de taquin moves on such tableaux. Using this notion, we give a simple bijective proof of the formula above for $\chi^{\lambda}_{2,1^{n-2}}$ by counting corresponding ribbon tableaux. To the author's knowledge, no such bijective proof is known before. Inspired by this simple proof, in Section~\ref{sec:chara-eval} we investigate the possibility of using trace forest to give bijective proof of more involved character evaluation formulae, which is equivalent to counting certain ribbon tableaux, and for this purpose we sketch a general scheme using structural induction on the tree structure of trace forest. This scheme leads to combinatorial proofs of the other two character evaluation formula above, for $\chi^{\lambda}_{3,1^{n-2}}$ and $\chi^{\lambda}_{4,1^{n-2}}$. Further possible development of this scheme is also discussed.

\section{Preliminaries} \label{sec:prelim}

\subsection{Partitions and standard tableaux}

A \emph{partition} $\lambda$ is a finite non-increasing sequence $(\lambda_i)_{i>0}$ of positive integers. We say that $\lambda$ is a partition of $n$ (noted as $\lambda \vdash n$) if $\sum_{i} \lambda_i = n$. The \emph{Ferrers diagram} of a partition $\lambda$ (also noted as $\lambda$ by abuse of notation) is a graphical representation of $\lambda$ consisting of left-aligned rows of boxes (also called \emph{cells}), in which the $i$-th line has $\lambda_i$ boxes. We assume that cells are all unit squares, and the center of the first cell in the first row is the origin of the plane. For a cell $w$ whose center is in $(i,j)$, we define its \emph{content} to be $c(w)=i-j$. Figure~\ref{fig:tableau} gives an example of a Ferrers diagram, drawn in French convention, with the content for each cell. 

A \emph{standard tableau} of the shape $\lambda \vdash n$ is a filling of the Ferrers diagram of $\lambda$ using integers from $1$ to $n$ such that each number is used exactly once, with increasing rows and columns. Figure~\ref{fig:tableau} also gives an example of a standard tableau. We note by $f^{\lambda}$ the number of standard tableaux of the form $\lambda$, and it is also the dimension of the irreducible representation of the symmetric group indexed by $\lambda$ (\textit{c.f.} \cite{sagan2001symmetric, vershik2004new}).

\begin{figure} 
\begin{multicols}{4}
\centering
\begin{tikzpicture}
\clip (-0.05,-0.05) rectangle (2.55,2.05);
\draw (0,0) \mysqr{0};
\draw (0.5,0) \mysqr{1};
\draw (1,0) \mysqr{2};
\draw (1.5,0) \mysqr{3};
\draw (2,0) \mysqr{4};
\draw (0,0.5) \mysqr{-1};
\draw (0.5,0.5) \mysqr{0};
\draw (1,0.5) \mysqr{1};
\draw (0,1) \mysqr{-2};
\draw (0.5,1) \mysqr{-1};
\draw (1,1) \mysqr{0};
\draw (0,1.5) \mysqr{-3};
\draw (0.5,1.5) \mysqr{-2};
\end{tikzpicture}

(a)

\columnbreak
\begin{tikzpicture}
\clip (-0.05,-0.05) rectangle (2.55,2.05);
\draw (0,0) \mysqr{1};
\draw (0.5,0) \mysqr{2};
\draw (1,0) \mysqr{5};
\draw (1.5,0) \mysqr{9};
\draw (2,0) \mysqr{11};
\draw (0,0.5) \mysqr{3};
\draw (0.5,0.5) \mysqr{7};
\draw (1,0.5) \mysqr{10};
\draw (0,1) \mysqr{4};
\draw (0.5,1) \mysqr{8};
\draw (1,1) \mysqr{13};
\draw (0,1.5) \mysqr{6};
\draw (0.5,1.5) \mysqr{12};
\end{tikzpicture}

(b)

\columnbreak
\begin{tikzpicture}
\clip (-0.05,-0.05) rectangle (2.55,2.05);
\draw (1.5,0) \mysqr{};
\draw (2,0) \mysqr{};
\draw (1,0.5) \mysqr{};
\draw (0,1) \mysqr{};
\draw (0.5,1) \mysqr{};
\draw (1,1) \mysqr{};
\draw (0,1.5) \mysqr{};
\draw (0.5,1.5) \mysqr{};
\end{tikzpicture}

(c)

\columnbreak
\begin{tikzpicture}
\clip (-0.05,-0.05) rectangle (2.55,2.05);
\draw (1.5,0) \mysqr{4};
\draw (2,0) \mysqr{6};
\draw (1,0.5) \mysqr{5};
\draw (0,1) \mysqr{1};
\draw (0.5,1) \mysqr{3};
\draw (1,1) \mysqr{8};
\draw (0,1.5) \mysqr{2};
\draw (0.5,1.5) \mysqr{7};
\end{tikzpicture}

(d)
\end{multicols}
\caption{(a) the Ferrers diagram of the partition $(5,3,3,2)$, with the content for each cell. (b) a standard tableau of the shape $(5,3,3,2)$. (c) the skew diagram of the skew partition $(5,3,3,2) / (3,2)$. (d) a skew tableau of the shape $(5,3,3,2) / (3,2)$.\label{fig:tableau}} 
\end{figure}
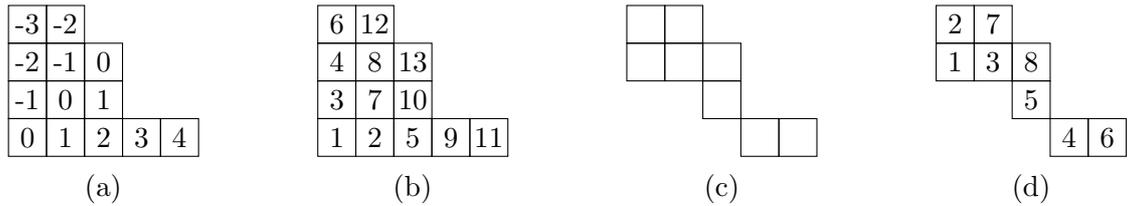

Definitions above can be generalized to so-called skew-partitions. A \emph{skew-partition} $\lambda/\mu$ is a pair of partitions $(\lambda, \mu)$ such that for all $i>0$, $\lambda_i \geq \mu_i$. Graphically, it is equivalent to that the Ferrers diagram of $\lambda$ covers totally that of $\mu$. We then define the \emph{skew diagram} of the from $\lambda/\mu$ as the difference of the Ferrers diagrams of $\lambda$ and of $\mu$, \textit{i.e.} the Ferrers diagram of $\lambda$ without cells that also appear in that of $\mu$. Figure~\ref{fig:tableau} gives an example of a skew diagram.

We now define the counterpart of standard tableau on skew diagrams. A \emph{skew tableau} of shape $\lambda/\mu$ is a filling of the skew diagram of $\lambda/\mu$ with $n$ cells that satisfies all conditions for standard tableaux. Figure~\ref{fig:tableau} gives an example of skew tableau. We note by $f^{\lambda/\mu}$ the number of skew tableaux with shape $\lambda/\mu$.

Standard tableaux and skew tableaux are classical combinatorial objects closely related to the representation theory of the symmetric group. In \cite{vershik2004new, sagan2001symmetric, stanley2001enumerative} details of this relation are described.

\subsection{Ribbon tableaux and the Murnaghan-Nakayama rule}

We note $S_n$ the symmetric group formed by permutations of $n$ elements. Let $\lambda, \mu$ be partitions of $n$, we note by $\chi^{\lambda}_{\mu}$ the \emph{irreducible character} of $S_n$ indexed by $\lambda$ evaluated on the conjugacy class indexed by $\mu$.

Irreducible characters can be expressed in a combinatorial way using the so-called ribbon tableaux. A \emph{ribbon} is a special skew diagram that is connected and without any $2 \times 2$ cells. The \emph{height} $ht(\lambda/\mu)$ of a ribbon $\lambda/\mu$ is the number of rows it spans minus one. A \emph{ribbon tableau} $T$ of the shape $\lambda$ is a sequence of partitions $\varnothing = \lambda^{(0)}, \lambda^{(1)}, \ldots, \lambda^{(k)} = \lambda$ such that $\lambda^{(i)} / \lambda^{(i-1)}$ is a ribbon for all $i>0$. The \emph{entry sequence} of $T$ is $(a_1, a_2, \ldots, a_k)$ with $a_i$ the number of cells in $\lambda^{(i)} / \lambda^{(i-1)}$. The total height of $T$ is defined by $ht(T)=\sum_i ht(\lambda^{(i)} / \lambda^{(i-1)})$, and the sign of $T$ is defined by $\sgn(T)=(-1)^{ht(T)}$. Figure~\ref{fig:ribbon-tableau} gives an example of a ribbon and a ribbon tableau. 

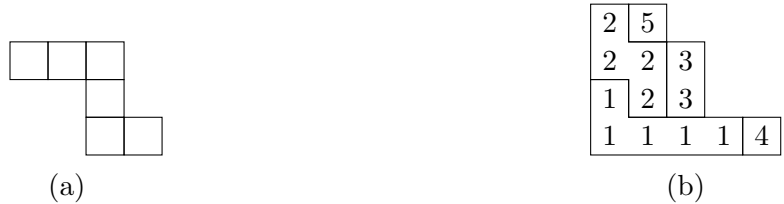
\begin{figure}
\begin{multicols}{2}
\centering
\begin{tikzpicture}
\clip (-0.05,-0.05) rectangle (2.55,2.05);
\draw (1.5,0) \mysqr{};
\draw (2,0) \mysqr{};
\draw (1.5,0.5) \mysqr{};
\draw (1.5,1) \mysqr{};
\draw (1,1) \mysqr{};
\draw (0.5,1) \mysqr{};
\end{tikzpicture}

(a)

\columnbreak
\begin{tikzpicture}
\clip (-0.05,-0.05) rectangle (2.55,2.05);
\draw (0,0) -- (2,0) -- (2,0.5) -- (0.5,0.5) -- (0.5,1) -- (0,1) -- cycle;
\path (0.25,0.25) node {1};
\path (0.75,0.25) node {1};
\path (1.25,0.25) node {1};
\path (1.75,0.25) node {1};
\path (0.25,0.75) node {1};
\draw (0.5,0.5) -- (1,0.5) -- (1,1.5) -- (0.5,1.5) -- (0.5,2) -- (0,2) -- (0,1) -- (0.5,1) -- cycle;
\path (0.75,0.75) node {2};
\path (0.75,1.25) node {2};
\path (0.25,1.25) node {2};
\path (0.25,1.75) node {2};
\draw (1,0.5) -- (1.5,0.5) -- (1.5,1.5) -- (1,1.5) -- cycle;
\path (1.25,0.75) node {3};
\path (1.25,1.25) node {3};
\draw (2,0) -- (2.5,0) -- (2.5,0.5) -- (2,0.5) -- cycle;
\path (2.25,0.25) node {4};
\draw (0.5,1.5) -- (1,1.5) -- (1,2) -- (0.5,2) -- cycle;
\path (0.75,1.75) node {5};
\end{tikzpicture}

(b)

\columnbreak

\end{multicols}
\caption{(a) the ribbon $(5,4,4) / (3,3,1)$ of height $2$. (b) a ribbon tableau $T$ of shape $(5,3,3,2)$ and of entry sequence $(5,4,2,1,1)$, with $\protect\sgn(T)=1$.\label{fig:ribbon-tableau}} 
\end{figure}

The Murnaghan-Nakayama rule (\textit{c.f.} Chapter~7.17 of \cite{stanley2001enumerative}) is a combinatorial interpretation of the irreducible character. According to this rule, we have $\chi^{\lambda}_{\mu}=\sum_T \sgn(T)$, where we sum over all ribbon tableau $T$ of the shape $\lambda$ and with entry sequence $\mu$. 

For a partition $\mu \vdash k$ and an integer $n > k$, we denote by $\mu, 1^{n-k}$ the partition obtained by concatenating $\mu$ with $n-k$ parts of size $1$. In this article, for a fixed ``small'' partition $\mu \vdash k$, we are interested by the evaluation of $\chi^{\lambda}_{\mu, 1^{n-k}}$ for arbitrary $\lambda \vdash n$ in terms of contents, which involves ribbon tableaux of shape $\lambda$ and entry sequence $\mu, 1^{n-k}$.

\begin{lem}[\textit{c.f.} \cite{corteel2004content}] \label{lem:character-as-skew-tableau}
For partitions $\lambda \vdash n$, $\mu \vdash k$ and $n > k$, we have
\[ \chi^{\lambda}_{\mu, 1^{n-k}} = \sum_{\nu \vdash k} f^{\lambda / \nu} \chi^{\nu}_{\mu}. \]
\end{lem}
\begin{proof}
Let $T_0$ be a ribbon tableaux of shape $\lambda$ and entry sequence $\mu, 1^{n-k}$. By retaining only the last $n-k$ ribbons of size $1$ in $T_0$, we obtain a skew tableau $T_1$, and $T = T_0 \setminus T_1$ is a ribbon tableau of entry sequence $\mu$. This is clearly a bijection between $T_0$ and $(T_1, T)$. Moreover, $\sgn(T)=\sgn(T_0)$. We now sum over the sign of all $T_0$ in bijection with $(T_1, T)$, first by the shape $\nu$ of $T$, then by each $T_1$ of shape $\lambda / \nu$, and finally by each $T$, and we finish the proof by the Murnaghan-Nakayama rule.
\end{proof}

By this lemma and the fact that irreducible characters span linearly the space of class functions (\textit{c.f.} Chapter~2.6 of \cite{serre1977linear}), character evaluation is equivalent to computing the number of skew tableaux of a certain shape. It is thus interesting for us to study skew tableaux.

\subsection{Jeu de taquin}

The jeu de taquin is a bijection between skew tableaux of different shapes. It was first introduced by Sch\"utzenberger and proved itself to be a powerful tool in the combinatorial representation theory of the symmetric group. Its applications includes the Sch\"utzenberger involution, the Littlewood-Richardson rule (\textit{c.f} \cite{stanley2001enumerative} for both), and also a bijective proof of Stanley's hook formula (\textit{c.f.} \cite{krattenthaler1999another}). An introduction to the jeu de taquin can be found in the Appendix~A of \cite{stanley2001enumerative}.

We now define the building block of the jeu de taquin on skew tableau, which are local exchanges of entries in the tableaux. Given a skew tableau $T$ with a distinguished entry $*$, the \emph{in-coming step} tries to permute $*$ with one of its ``inward'' neighbors, the ones immediately below or to the left, while conserving the increasing conditions of skew tableau. This is always possible as in the left side of Figure~\ref{fig:jeu-de-taquin}. The \emph{out-going step} is similarly defined, by exchange with entries immediately above or to the right. Figure~\ref{fig:jeu-de-taquin} illustrates the precise rule of both kinds of steps. We verify that in-coming steps are exactly the reverse of out-going steps. 

We now define the \emph{in-coming slide} of the distinguished entry $*$ as successive applications of the in-coming steps to $*$ until it no longer has neighbor below or to the left. Since in-coming steps are reversible, given the distinguished entry and the resulting skew tableau, we can also reverse an in-coming slide. Therefore, the in-coming slide, which is a global operation on tableaux, is also reversible.

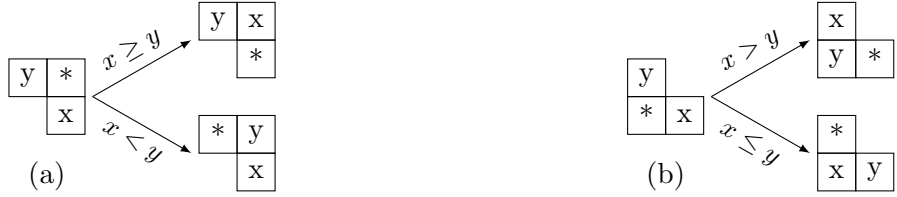
\begin{figure}
\begin{multicols}{2}
\centering
\begin{tikzpicture}
\draw (0,-0.5) \mysqr{x};
\draw (-0.5,0) \mysqr{y};
\draw (0,0) \mysqr{*};

\draw (2.5,0.25) \mysqr{*};
\draw (2,0.75) \mysqr{y};
\draw (2.5,0.75) \mysqr{x};

\draw (2.5,-1.25) \mysqr{x};
\draw (2,-0.75) \mysqr{*};
\draw (2.5,-0.75) \mysqr{y};

\draw[-latex] (0.6, 0) -- (1.9, 0.75) node[midway, above, sloped] {\small $x \geq y$};
\draw[-latex] (0.6, 0) -- (1.9, -0.75) node[midway, below, sloped] {\small $x<y$};

\draw (0,-1.05) node {(a)};
\end{tikzpicture}

\begin{tikzpicture}
\draw (0,-0.5) \mysqr{x};
\draw (-0.5,0) \mysqr{y};
\draw (-0.5,-0.5) \mysqr{*};

\draw (2.5,0.25) \mysqr{*};
\draw (2,0.75) \mysqr{x};
\draw (2,0.25) \mysqr{y};

\draw (2.5,-1.25) \mysqr{y};
\draw (2,-0.75) \mysqr{*};
\draw (2,-1.25) \mysqr{x};

\draw[-latex] (0.6, 0) -- (1.9, 0.75) node[midway, above, sloped] {\small $x > y$};
\draw[-latex] (0.6, 0) -- (1.9, -0.75) node[midway, below, sloped] {\small $x \leq y$};

\draw (0,-1.05) node {(b)};
\end{tikzpicture}
\end{multicols}
\caption{In-coming step (a) and out-going step (b) in jeu de taquin\label{fig:jeu-de-taquin}} 
\end{figure}

We now give a bijection that relates standard tableaux and skew tableaux using the jeu de taquin.

\begin{lem} \label{lem:skew-bijection}
For a partition $\lambda \vdash n$ and an integer $k > 0$, The jeu de taquin gives a bijection between the following two sets:

\begin{itemize}
\item the set of tuples $(T, a_1, \ldots, a_k)$, where $T$ is a standard tableau of shape $\lambda$, and all $a_i$ distinct integers between $1$ and $n$,
\item the set of tuples $(T_0, T_1, a_1, \ldots, a_k)$, where $T_0$ is a skew tableau of shape $\lambda / \mu$ for a certain partition $\mu \vdash k$ of entries from $1$ to $n-k$, $T_1$ a standard tableau of shape $\mu$ of entries from $1$ to $k$, and all $a_i$ distinct integers between $1$ and $n$.
\end{itemize}
\end{lem}
\begin{proof}
We apply the in-coming slide to $a_1, \ldots, a_k$ successively on $T$. We then obtain a skew tableau $T_0'$ of shape $\lambda / \mu$ for a certain partition $\mu \vdash k$ and a standard tableau $T_1$ of shape $\mu$ of entries from $1$ to $k$ that indicates the exclusion order of cells. The entries in $T_0'$ are all integers from $1$ to $n$ except all $a_i$, but since all the $a_i$ are known, we can renumber entries in $T_0'$ to produce a standard tableau of entries from $1$ to $n-k$, and the reconstruction from $T_0$ to $T_0'$ is easy given all $a_i$. Since in-coming slides are reversible, given the $(a_1, \ldots, a_k), T_0, T_1$, we can reconstruct $T$. We conclude that it is indeed a bijection. Figure~\ref{fig:skew-bijection} gives an example for $\lambda = (5,3,3,2) , \mu=(2)$
\end{proof}

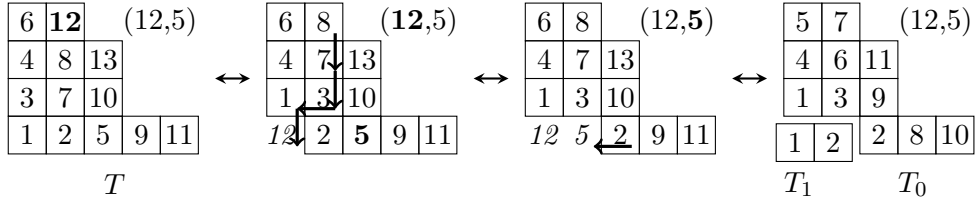
\begin{figure}
\centering
\begin{tikzpicture}
\begin{scope}
\draw (0,0) \mysqr{1};
\draw (0.5,0) \mysqr{2};
\draw (1,0) \mysqr{5};
\draw (1.5,0) \mysqr{9};
\draw (2,0) \mysqr{11};
\draw (0,0.5) \mysqr{3};
\draw (0.5,0.5) \mysqr{7};
\draw (1,0.5) \mysqr{10};
\draw (0,1) \mysqr{4};
\draw (0.5,1) \mysqr{8};
\draw (1,1) \mysqr{13};
\draw (0,1.5) \mysqr{6};
\draw (0.5,1.5) \mysqr{\textbf{12}};
\draw (2, 1.75) node {(12,5)};
\draw (1.4, -0.4) node {$T$};
\end{scope}

\begin{scope}[xshift=3.4cm]
\draw[very thick, ->] (0.9, 1.6) -- (0.9, 1.1); 
\draw[very thick, ->] (0.9, 1.1) -- (0.9, 0.6);
\draw[very thick, ->] (0.9, 0.6) -- (0.4, 0.6);
\draw[very thick, ->] (0.4, 0.6) -- (0.4, 0.1);

\draw (0.2,0.25) node{\textit{12}};
\draw (0.5,0) \mysqr{2};
\draw (1,0) \mysqr{\textbf{5}};
\draw (1.5,0) \mysqr{9};
\draw (2,0) \mysqr{11};
\draw (0,0.5) \mysqr{1};
\draw (0.5,0.5) \mysqr{3};
\draw (1,0.5) \mysqr{10};
\draw (0,1) \mysqr{4};
\draw (0.5,1) \mysqr{7};
\draw (1,1) \mysqr{13};
\draw (0,1.5) \mysqr{6};
\draw (0.5,1.5) \mysqr{8};
\draw (2, 1.75) node {(\textbf{12},5)};
\end{scope}

\begin{scope}[xshift=6.8cm]
\draw[very thick, ->] (1.4,0.1) -- (0.9, 0.1);

\draw (0.25,0.25) node{\textit{12}};
\draw (0.75,0.25) node{\textit{5}};

\draw (1,0) \mysqr{2};
\draw (1.5,0) \mysqr{9};
\draw (2,0) \mysqr{11};
\draw (0,0.5) \mysqr{1};
\draw (0.5,0.5) \mysqr{3};
\draw (1,0.5) \mysqr{10};
\draw (0,1) \mysqr{4};
\draw (0.5,1) \mysqr{7};
\draw (1,1) \mysqr{13};
\draw (0,1.5) \mysqr{6};
\draw (0.5,1.5) \mysqr{8};
\draw (2, 1.75) node {(12,\textbf{5})};
\end{scope}

\begin{scope}[xshift=10.2cm]
\draw (-0.1,-0.1) \mysqr{1};
\draw (0.4,-0.1) \mysqr{2};

\draw (1,0) \mysqr{2};
\draw (1.5,0) \mysqr{8};
\draw (2,0) \mysqr{10};
\draw (0,0.5) \mysqr{1};
\draw (0.5,0.5) \mysqr{3};
\draw (1,0.5) \mysqr{9};
\draw (0,1) \mysqr{4};
\draw (0.5,1) \mysqr{6};
\draw (1,1) \mysqr{11};
\draw (0,1.5) \mysqr{5};
\draw (0.5,1.5) \mysqr{7};
\draw (2, 1.75) node {(12,5)};

\draw (0.2, -0.4) node{$T_1$};
\draw (1.7, -0.4) node{$T_0$};
\end{scope}

\draw[thick,<->,shorten >=2pt,shorten <=2pt,>=stealth] (2.65, 1) -- (3.25,1);
\draw[thick,<->,shorten >=2pt,shorten <=2pt,>=stealth] (6.05, 1) -- (6.65,1);
\draw[thick,<->,shorten >=2pt,shorten <=2pt,>=stealth] (9.45, 1) -- (10.05,1);
\end{tikzpicture}
\caption{Example of bijection relation standard tableaux and skew tableaux via the jeu de taquin\label{fig:skew-bijection}} 
\end{figure}

From the proof of the lemma above, we can conclude that, to calculate a certain $f^{\lambda / \mu}$ for $\mu \vdash k$, it suffices to count the number of tuples $(T, (a_1, \ldots, a_k))$, with $T$ a standard tableau of shape $\lambda$, that are associated to $(T_0, T_1, (a_1, \ldots, a_k))$, with $T_0$ of shape $\lambda / \mu$ via the jeu de taquin. To accomplish this task, we need to know more about the fine structure of the jeu de taquin.

\section{Trace forest of jeu de taquin} \label{sec:jdt}

We will now define a structure related to the jeu de taquin in skew tableaux called ``trace forest''. It is essentially a directed graph whose vertices are cells in the tableau, and it encapsulates the trace of the in-coming slide of each entries.

\begin{defi}
Given a skew tableau $T$, we define its \emph{trace forest}, which is a directed graph $F$ with cells in $T$ as vertices, as follows. For a cell $w$ in $T$ with neighbors below or to the left and $a$ its entry, we point an arc from $w$ to the destination of in the in-coming step for $a$. It is clear that no cycle can exist, thus the constructed graph is a forest, rooted at cells without neighbor below or to the left.
\end{defi}

Figure~\ref{fig:trace-forest} gives some examples of skew tableaux and their trace forests. For a skew tableau $T$, let $F$ be its trace forest. By definition, the in-coming step with any cell $c \in T$ follows exactly the arc from $c$ in the trace forest. With simple induction on $F$, we can see that the in-coming slide of the entry of any cell $c \in T$ coincides with the path from $c$ to its root in $F$, which gives the structure $F$ the name ``trace forest''.

\begin{figure}
\centering
\begin{tikzpicture}
\begin{scope}
\draw (0,0) \mysqr{1};
\draw (0.5,0) \mysqr{2};
\draw (1,0) \mysqr{5};
\draw (1.5,0) \mysqr{9};
\draw (2,0) \mysqr{11};
\draw (0,0.5) \mysqr{3};
\draw (0.5,0.5) \mysqr{7};
\draw (1,0.5) \mysqr{10};
\draw (0,1) \mysqr{4};
\draw (0.5,1) \mysqr{8};
\draw (1,1) \mysqr{13};
\draw (0,1.5) \mysqr{6};
\draw (0.5,1.5) \mysqr{12};
\end{scope}

\begin{scope}[xshift=3cm]
\draw (0,0) \mysqr{};
\draw (0.5,0) \mysqr{};
\draw (1,0) \mysqr{};
\draw (1.5,0) \mysqr{};
\draw (2,0) \mysqr{};
\draw (0,0.5) \mysqr{};
\draw (0.5,0.5) \mysqr{};
\draw (1,0.5) \mysqr{};
\draw (0,1) \mysqr{};
\draw (0.5,1) \mysqr{};
\draw (1,1) \mysqr{};
\draw (0,1.5) \mysqr{};
\draw (0.5,1.5) \mysqr{};

\draw[very thick] (0.25, 1.75) -- (0.25,0.25);
\draw[very thick] (2.25,0.25) -- (0.25,0.25);
\draw[very thick] (0.75,1.75) -- (0.75,0.75) -- (0.25, 0.75);
\draw[very thick] (1.25, 1.25) -- (1.25, 0.75) -- (0.75, 0.75);
\end{scope}

\begin{scope}[xshift=6cm]
\draw (1,0) \mysqr{5};
\draw (1.5,0) \mysqr{9};
\draw (2,0) \mysqr{10};
\draw (0.5,0.5) \mysqr{3};
\draw (1,0.5) \mysqr{4};
\draw (0,1) \mysqr{1};
\draw (0.5,1) \mysqr{6};
\draw (1,1) \mysqr{7};
\draw (0,1.5) \mysqr{2};
\draw (0.5,1.5) \mysqr{8};
\end{scope}

\begin{scope}[xshift=9cm]
\draw (1,0) \mysqr{};
\draw (1.5,0) \mysqr{};
\draw (2,0) \mysqr{};
\draw (0.5,0.5) \mysqr{};
\draw (1,0.5) \mysqr{};
\draw (0,1) \mysqr{};
\draw (0.5,1) \mysqr{};
\draw (1,1) \mysqr{};
\draw (0,1.5) \mysqr{};
\draw (0.5,1.5) \mysqr{};

\draw[very thick] (0.25, 1.75) -- (0.25, 1.25);
\draw[very thick] (0.75, 1.75) -- (0.75, 0.75);
\draw[very thick] (1.25, 1.25) -- (0.75, 1.25);
\draw[very thick] (2.25, 0.25) -- (1.25, 0.25) -- (1.25, 0.75);
\end{scope}
\end{tikzpicture}
\caption{Examples of skew tableaux and their trace forests\label{fig:trace-forest}} 
\end{figure}
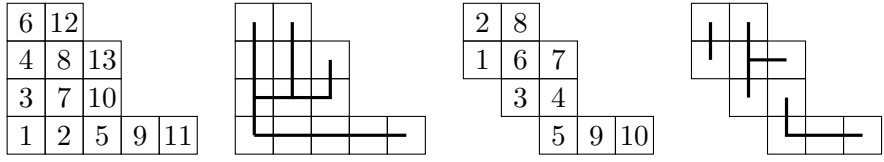

We now study how an in-coming slide changes the trace forest of a skew tableau. We begin with some definitions. For a cell $c$ in $F$, we call its child to the right the \emph{right child}, and its child above the \emph{upper child}, noted as $c_<$ and $c_\vee$. We note by $F_<$ and $F_\vee$ the subtree of $F$ rooted in $c_<$ and $c_\vee$ respectively. 

Let $T$ be a skew tableau, $S$ a subtree of its trace forest rooted in $a$ and $a_{<}, a_{\vee}$ its right and upper child (if they exist). For a cell $c \in S$, we note $T^{c}$ the tableau obtained by applying an in-coming slide on $c$, and $F^{c}$ its trace forest. The cells in $S \setminus \{ a \}$ are partitioned into the following categories, as in Figure~\ref{fig:fine-structures}:

\begin{itemize}
\item $D_1(c)$ (resp. $D_2(c)$), the subtree rooted at the right child (resp. the upper child) of $c$;
\item $P_1(c)$ (resp. $P_2(c)$), the set of ancestors of $c$ (including $c$) that issue a horizontal (resp. vertical) arc;
\item $R(c)$ (resp. $A(c)$), the set of cells not in categories above and whose in-coming slide path lies below (resp. above) that of $c$.
\end{itemize}

We note $C_<(c,S) = D_1(c) \cup P_1(c) \cup R(c)$ and $C_\vee(c,S) = D_2(c) \cup P_2(c) \cup A(c)$. We can see that $C_<(c,S)$ and $C_\vee(c,S)$ divide cells in $S \setminus \{ a \}$ into two groups. In the following lemma, we see that this grouping of cells is related to the structure of $F^{c}$ after the in-coming slide of $c$ applied to $T$.

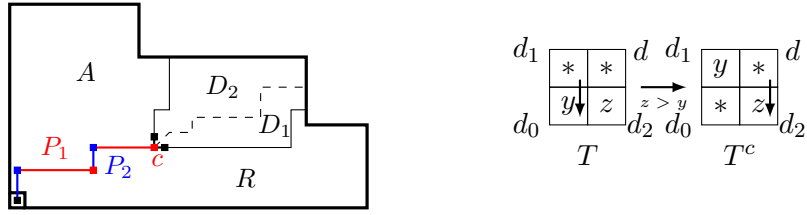
\begin{figure}
\centering
\begin{tikzpicture}

\begin{scope}
\draw[very thick] (-0.1, 2.6) -- (-0.1, 0.1) -- (0.1, 0.1) -- (0.1,-0.1) -- (4.6, -0.1) -- (4.6, 1) -- (3.8, 1) -- (3.8, 1.9) -- (1.6, 1.9) -- (1.6, 2.6) -- cycle;
\draw[very thick] (-0.1,0.1) -- (-0.1,-0.1) -- (0.1,-0.1);
\draw[thick, blue] (0,0) -- (0,0.4);
\draw[thick, red] (0,0.4) -- node[above, midway]{\small \color{red} $P_1$} (1,0.4);
\draw[thick, blue] (1,0.4) -- (1, 0.7) node[right, very near start]{\small \color{blue} $P_2$};
\draw[thick, red] (1,0.7) -- (1.8, 0.7);
\draw[thick, black] (1.8, 0.84) -- (1.8, 0.7) -- (1.94, 0.7);
\draw[dashed] (1.8, 0.7) -- (2, 0.9) -- (2.3, 0.9) -- (2.3, 1.1) -- (3.2, 1.1) -- (3.2, 1.5) -- (3.8, 1.5);
\draw (1.8, 0.84) -- (1.8, 1.2) -- (2, 1.2) -- (2, 1.9);
\draw (1.94, 0.7) -- node[above, very near end]{\small $D_1$} (3.6, 0.7) -- (3.6, 1.2) -- (3.8, 1.2);

\filldraw[black] (-0.04,-0.04) rectangle +(0.08, 0.08);
\filldraw[blue] (-0.04,0.36) rectangle +(0.08, 0.08);
\filldraw[red] (0.96,0.36) rectangle +(0.08, 0.08);
\filldraw[blue] (0.96,0.66) rectangle +(0.08, 0.08);
\filldraw[red] (1.76,0.66) rectangle +(0.08, 0.08) node[below]{\small $c$};
\filldraw[black] (1.9,0.66) rectangle +(0.08, 0.08);
\filldraw[black] (1.76,0.8) rectangle +(0.08, 0.08);
\draw (3, 0.3) node{\small $R$};
\draw (0.9, 1.7) node{\small $A$};
\draw (2.7, 1.5) node{\small $D_2$};
\end{scope}

\begin{scope}[xshift=7cm, yshift=1cm]
\draw (0,0) \mysqr{$y$};
\draw (0,0.5) \mysqr{$*$};
\draw (0.5,0) \mysqr{$z$};
\draw (0.5,0.5) \mysqr{$*$};
\draw[-latex, thick] (0.4,0.6) -- (0.4,0.1);
\draw (0.5,-0.4) node{$T$};
\draw (1.2, 1) node{$d$};
\draw (1.2, 0) node{$d_2$};
\draw (-0.3,0) node{$d_0$};
\draw (-0.3,1) node{$d_1$};

\draw (2,0) \mysqr{$*$};
\draw (2,0.5) \mysqr{$y$};
\draw (2.5,0) \mysqr{$z$};
\draw (2.5,0.5) \mysqr{$*$};
\draw[-latex, thick] (2.9,0.6) -- (2.9,0.1);
\draw (2.5,-0.4) node{$T^c$};
\draw (3.2, 1) node{$d$};
\draw (3.2, 0) node{$d_2$};
\draw (1.7, 0) node{$d_0$};
\draw (1.7, 1) node{$d_1$};

\draw[-latex, thick] (1.2,0.5) -- node[below, midway]{\tiny $z>y$} (1.8,0.5);
\end{scope}
\end{tikzpicture}
\caption{Fine structure of the trace forest\label{fig:fine-structures}} 
\end{figure}

\begin{lem} \label{lem:trace-forest-separation}
For a skew tableau $T$, let $S$ be a subtree in its trace forest, and $c \in S$. For $d \in C_<(c,S)$ (resp. $d \in C_\vee(c,S)$), the in-coming slide path of $d$ in $T^{c}$ lies to the right (resp. above) of that of $c$ in $T$.
\end{lem}
\begin{proof}
We only need to show that no arc goes between elements in $C_<(c,S)$ and $C_\vee(c,S)$ in the trace forest of $T^{c}$, and it will entail the lemma because of the relative position of $C_<(c,S)$ and $C_\vee(c,S)$. We will first prove that there is no arc from $C_<(c,S)$ to $C_\vee(c,S)$ in the trace forest of $T^{c}$. Let $d \in C_<(c,S)$ and $x$ the entry of $d$ in $T^{c}$, $d_1$ be the cell immediately to the left of $d$, $d_2$ the one below $d$ and $d_0$ the one on the south-west. There are three cases: $d \in P_1(c)$, $d \in D_1(c)$ and $d \in R(c)$. 

For $d \in P_1(c)$ and $d \in R(c)$, the only possible way that $d_1 \in C_\vee(c,S)$ is the case $d_1 \in P_2(c)$. For $d \in D_1(c)$, it suffices to prove for the root $d=r_1$ of $D_1(c)$, and the only possible way that $d_1 \in C_\vee(c,S)$ is still $d_1 \in P_2(c)$. Therefore, in all 3 cases, we have that the arc of $d_1$ points to $d_0$ in $F$.

Let $y$ be the entry in $d_1$ and $z$ in $d_2$ in $T^{c}$. By definition of $P_2(c)$, $d_0$ contains $y$, and it entails $y < z$ by the definition of skew tableau. Therefore in $T^{c}$, the arc from $d$ points to $d_2$ according to the rule of the jeu de taquin, and we have the wanted separation. The right side of Figure~\ref{fig:fine-structures} illustrates this argument.

The proof that there is no arc from $C_\vee(c,S)$ to $C_<(c,S)$ in the trace forest of $T^{c}$ is similar.
\end{proof}

Lemma~\ref{lem:trace-forest-separation} can be seen as a clarification of an argument in Lemma~$\mathrm{HC}^{*}$ in \cite{krattenthaler1999another}. Using Lemma~\ref{lem:trace-forest-separation}, we have the following simple bijective proof of a well-known character formula (\textit{c.f.} \cite{ingram1950some, corteel2004content, lassalle2008explicit}). To the knowledge of the author, no purely bijective proof is known before for this simple formula.

\begin{thm} \label{thm:mu-2}
For a partition $\lambda \vdash n$, we identify $\lambda$ and its Ferrers diagram, and we have
\[ n(n-1)\chi^{\lambda}_{(2,1^{n-2})} = 2 f^{\lambda} \sum_{w \in \lambda} c(w). \]
\end{thm}
\begin{proof}
Since from Lemma~\ref{lem:character-as-skew-tableau} follows $\chi^{\lambda}_{(2,1^{n-2})} = f^{\lambda / (2)} - f^{\lambda / (1,1)}$, we want to count the difference between the number of skew tableaux of shape $\lambda / (2)$ and those of shape $\lambda / (1,1)$.

Let $(T,a,b)$ be a tuple with $T$ standard tableau of shape $\lambda$ and $a \neq b$ two entries in $T$. We let $F$ denote the only tree in the trace forest of $T$, and we let $T_1(T,a,b)$ denote the skew tableau $T_1$ such that $(T,a,b)$ is associated to $(T_1,a,b)$ in the bijection in Lemma~\ref{lem:skew-bijection}, in which $T_0$ is fixed in our case. Therefore, when going through all $(T,a,b)$, $T_1(T,a,b)$ goes over each skew tableau of shape $\lambda / (2)$ or $\lambda / (1,1)$ exactly $n(n-1)$ times.

For entries $a,b$ with $a<b$, we consider the contribution of $T_1(T,a,b), T_1(T,b,a)$ to $f^{\lambda / (2)} - f^{\lambda / (1,1)}$.

If $a$ is not an ancestor of $b$ in $T$, they have a common ancestor $c$, and by symmetry we can suppose that $b$ is on the subtree rooted in the upper child of $c$. From Lemma~\ref{lem:trace-forest-separation}, we know that $b \in A(a) \subset C_\vee(a,F)$ in $T^a$ and $a \in R(b) \subset C_<(b,F)$ in $T^b$, therefore $T_1(T,a,b)$ is of shape $\lambda / (1,1)$, while $T_1(T,b,a)=0$ is of shape $\lambda / (2)$. Thus this case does not contribute to $f^{\lambda / (2)} - f^{\lambda / (1,1)}$.

The other case is that $a$ is an ancestor of $b$ in $T$. If the path from $b$ to $a$ ends with a horizontal arc pointing at $a$, then from Lemma~\ref{lem:trace-forest-separation}, we have $b \in D_1(a) \subset C_<(a,F)$ in $T^a$ and $a \in P_1(b) \subset C_<(b,F)$ in $T^b$, therefore $T_1(T,a,b)$ and $T_1(T,b,a)$ are both of shape $\lambda / (2)$. Otherwise, if the path from $b$ to $a$ ends with a vertical arc pointing at $a$, $T_1(T,a,b)$ and $T_1(T,b,a)$ are both of shape $\lambda / (1,1)$. The path in the trace forest from $b$ at $(i, j)$ to the cell at $(0,0)$ consists of $i$ horizontal arcs and $j$ vertical arcs. Therefore, if we sum over all ancestors $a$ of $b$, among all $T_1(T,a,b)$ and $T_1(T,b,a)$, we have $2i$ tableaux of shape $\lambda / (2)$ and $2j$ ones of shape $\lambda / (1,1)$, which gives a contribution of $2c(b)$ to $f^{\lambda / (2)} - f^{\lambda / (1,1)}$. This contribution is independent of $T$.

In the end, we have $n(n-1)(f^{\lambda / (2)} - f^{\lambda / (1,1)}) = 2 f^{\lambda} \sum_{w \in \lambda} c(w)$, thus finish the proof.
\end{proof}

In the proof above, there are two cases for entries $a,b$: the case where $a$ is not an ancestor of $b$ that contributes nothing, and the other case where contents appears naturally in the contribution. In the former case, $a,b$ play the same role, which reflects some kind of symmetry. For more general cases, we need to apply the jeu de taquin to several entries and count the skew tableaux obtained of a certain shape. It is thus desirable to extract similar symmetries. However, the task becomes monstrous when we pass to more entries due to case analysis. To surmount this difficulty, it is natural to try to use the tree structure of the trace forest to implicitly extract the symmetry we want.

\section{Character evaluation using trace forest} \label{sec:chara-eval}

We will now use the notion of trace forest to calculate $f^{\lambda / \mu}$ with fixed small $\mu$. In \cite{corteel2004content} and \cite{lassalle2008explicit} (see also \cite{kerov1994polynomial}), it was proved that $\chi^{\lambda}_{\mu}$ can be expressed using so-called ``content evaluation''. By Lemma~\ref{lem:character-as-skew-tableau}, we know that $f^{\lambda / \mu}$ can also be expressed by such content evaluation. It is now interesting to study the interaction between content evaluation and trace forest, and how it applies to character evaluation. 

In this section, we will define a notion called the ``inductive form'' of functions on any subtree in the trace forest. It enables the computation of such functions by identification of inductive form. We also give the inductive form of several content evaluations. Then we proceed to the bijective counting of skew tableaux of different shapes using the jeu de taquin, and by identification of inductive form, we obtain the expression of several $f^{\lambda / \mu}$ for general $\lambda$ and small $\mu$ in terms of content evaluation, which gives bijective proofs of various character evaluation formulae.

\subsection{Content powersums}

We will start by defining various content powersums on subtrees of the trace forest of skew-tableaux related to contents. 

For a skew tableau $T$, let $S$ be a subtree in its trace forest and $a$ its root. We denote by $c_{a}(w)$ the \emph{relative content} of a cell $w$ w.r.t. the root $a$, \textit{i.e.} $a$ is taken as the origin when computing the relative content $c_{a}(w)$. We have $c_{a}(w) = c(w) - c(a)$, where $c$ stands for the normal content. For any partition $\alpha=(\alpha_1, \ldots, \alpha_l) \vdash k$, we define \emph{content power sums} of $S$ denoted by $cp^{\alpha}(S)$ as follows, with the convention $0^0=1$.
\[ cp^{\alpha}(S) = \prod_{i=1}^{l} \sum_{w \in S} c_a(w)^{\alpha_i - 1} \]

For any standard tableau $T$ of shape $\lambda$ and $F_T$ its only tree in its trace forest, we denote $cp^{\alpha}(\lambda) = cp^{\alpha}(F_T)$.

The definition of $cp^{\alpha}$ extends readily to any subset of cells in $T$. For any subset of cells $C$ and any cell $a$, we define $cp^{\alpha}_{a}(C)$ as follows
\[ cp^{\alpha}_{a}(C) = \prod_{i=1}^{l} \sum_{w \in C} c_a(w)^{\alpha_i - 1} \]

We note that the subscript $a$ in $cp^{\alpha}_a$ represents the ``origin'' for the relative contents used in the function. When evaluated over a tree, we omit the subscript since we always take the root as origin. We notice that $cp^{(k)}$ is the sum of the content powersum of power $k-1$.

We can see that, for $\alpha = (\alpha_1, \ldots, \alpha_l)$ and $S$ a subtree of the trace forest of some standard tableau, the functions $cp^{\alpha}$ and $cp^{\alpha}_a$ are the powersum function $p_{(\alpha_1 - 1, \ldots, \alpha_l - 1)}$ evaluated over multisets of contents, multiplied by a polynomial in $|S|$. We recall that the powersum functions span linearly the algebra of symmetric functions denoted as $\Lambda$ (\textit{c.f.} \cite{stanley2001enumerative}, Chapter 7). Therefore, our $cp^{\alpha}_a$ also inherit an algebra structure for a fixed $a$, noted $\Lambda_c$. When evaluated on the whole standard tableau of shape $\lambda$, $\Lambda_c$ is exactly the nice algebra generated by the shifted symmetric functions (\textit{c.f.} \cite{corteel2004content, kerov1994polynomial}). We will now show that the set for all $a$ also form an algebra by showing that we can change the origin $a$.

We begin by some definitions. We recall that for a cell $a$ in a Young diagram, we note by $a_<$ the cell to its right and $a_\vee$ the cell above. We define two linear operators $\Gamma_+$ and $\Gamma_-$ as follows.
\[ \Gamma_+ cp^{(k)}_a = cp^{(k)}_{a_<}, \quad \Gamma_- cp^{(k)}_a = cp^{(k)}_{a_\vee} \]
By requiring $\Gamma_+$ and $\Gamma_-$ to be compatible with multiplication, \textit{i.e.} $\Gamma_+ (fg) = (\Gamma_+ f)(\Gamma_+ g)$ and the same for $\Gamma_-$, these two operators are thus defined over the whole $\Lambda_c$.

In fact, the algebra $\Lambda_c$ is stable by these operators $\Gamma_+$ and $\Gamma_-$ via the following lemma.

\begin{lem}\label{lem:operator-gamma}
For any integer $k \geq 1$, the result of application of $\Gamma_+$ and $\Gamma_-$ is as follows.
\[
\Gamma_+ cp^{(k)}_a = \sum_{i=0}^{k - 1} (-1)^{i} \binom{k-1}{i} cp^{(i)}_a, \quad \Gamma_- cp^{(k)}_a = \sum_{i=0}^{k-1} \binom{k-1}{i} cp^{(i)}_a
\]
Therefore $\Lambda_c$ is stable by $\Gamma_+$ and $\Gamma_-$. Moreover, $\Gamma_+ \Gamma_- = \Gamma_- \Gamma_+ = \mathrm{id}$.
\end{lem}
\begin{proof}
We have the simple observation that, for any cell $w, a$, we have $c_{a_<}(w) + 1 = c_a(w) = c_{a_\vee}(w) - 1$. This is simply due to the change of origin.

Now for any subset $C$ of cells, we have:
\[ (\Gamma_+ cp^{(k)}_a)(C) = \sum_{w \in C} c_{a_<}^{k-1}(w) = \sum_{w \in C} (c_{a}(w)-1)^{k-1} = \sum_{i=0}^{k-1} (-1)^{i} \binom{k-1}{i} cp^{(i)}_a(C) \]

\[ (\Gamma_- cp^{(k)}_a)(C) = \sum_{w \in C} c_{a_\vee}^{k-1}(w) = \sum_{w \in C} (c_{a}(w)+1)^{k-1} = \sum_{i=0}^{k-1} \binom{k-1}{i} cp^{(i)}_a(C) \]

For $\Gamma_+ \Gamma_- = \Gamma_- \Gamma_+ = \mathrm{id}$, we only need to notice that $c_{(a_<)_\vee}(w) = c_a(w) = c_{(a_\vee)_<}(w)$ for any cell $w$.
\end{proof}

We notice that, for a function $f \in \Lambda_c$, when viewed as a function on a partition $\lambda$, $f$ is a shifted symmetric function in $\lambda_1, \lambda_2, \ldots$. The fact that shifted symmetric functions form a nice algebra hints that $\Lambda_c$ also have a nice algebraic structure.

\subsection{Content evaluation and inductive form}

We will need some more definitions. For $S$ a subtree in a trace forest rooted in $r$ and any partition $\alpha \vdash k$, we use $\mycpr{\alpha}(S)$ (resp. $\mycpa{\alpha}(S)$) as follows.
\[ \mycpr{\alpha}(S) = cp^{\alpha}_r(S_<), \quad \mycpa{\alpha}(S) = cp^{\alpha}_r(S_\vee) \]

We now define a transformation called the \emph{inductive form}. For $S$ a subtree of a trace forest rooted in $a$, we note $S_<, S_\vee$ the subtrees rooted in $a_<$ and $a_\vee$. Let $f$ be a real-valued function on subtrees of a trace forest, its inductive form is defined by $(\Delta f)(S) = f(S) - f(S_<) - f(S_\vee)$. The transformation $\Delta$ is clearly linear.

\begin{lem} \label{lem:inductive-identify}
Let $f, g$ be two functions on subtrees of trace forests. If $f(\varnothing) = g(\varnothing)$ and $\Delta f = \Delta g$, then we have $f=g$.
\end{lem}
\begin{proof}
Since $\Delta$ is linear, for any $S$, $(f-g)(S) = (f-g)(S_<) + (f-g)(S_\vee)$, and we conclude the proof by structural induction.
\end{proof}

Now we compute the inductive form of $cp^{\alpha}$. It will be used later to identify characters, which can be seen as a function on the trace forest, as a sum of $cp^{\alpha}$.

\begin{prop} \label{prop:cp-inductive-form}
We have the following equalities for any subtree $S$ in a trace forest.
\begin{align}
cp^{(1)}(S) = 1 + \mycpr{1}(S) + \mycpa{1}(S); &\quad \forall k>1, cp^{(k)}(S) = \mycpr{k}(S) + \mycpa{k}(S) \\
cp^{(k)}(S_<) = \sum_{i=0}^{k-1} (-1)^{i} \binom{k-1}{i} \mycpr{k-i}(S); &\quad cp^{(k)}(S_\vee) = \sum_{i=0}^{k-1} \binom{k-1}{i} \mycpa{k-i}(S)
\end{align}

Furthermore, for any $\alpha \vdash d$, $\Delta cp^{\alpha}$ can be expressed as a polynomial in $\mycpr{k}$ and $\mycpa{k}$.
\end{prop}
\begin{proof}
The equalities in (1) comes directly from the definition of $cp^{(k)}$. For (2), we notice that $cp^{(k)}(S_<)$ is rooted in $a_<$, thus $cp^{(k)}(S_<) = (\Gamma_+ cp^{(k)}_a)(S_<)$, and we conclude by Lemma~\ref{lem:operator-gamma}. For $cp^{(k)}(S_\vee)$ we have similarly $cp^{(k)}(S_\vee) = (\Gamma_- cp^{(k)}_a)(S_\vee)$.

By (1) and (2), for any subtree $S$ of a trace forest, we can express $cp^{\alpha}(S)$, $cp^{\alpha}(S_<)$ and $cp^{\alpha}(S_\vee)$ as polynomials in $\mycpr{k}(S)$ and $\mycpa{k}(S)$. We finish the proof with the fact that $(\Delta cp^{\alpha})(S) = cp^{\alpha}(S) - cp^{\alpha}(S_<) - cp^{\alpha}(S_\vee)$.
\end{proof}

Here are some examples of the inductive form of some $cp^{\alpha}$. For simplicity, we consider $cp^{\alpha}, \mycpr{\alpha}, \mycpa{\alpha}$ as functions and omit their arguments.

\begin{align*}
\Delta cp^{(1)} &= 1, \quad \Delta cp^{(2)} = \mycpr{1} - \mycpa{1}, \quad \Delta cp^{(1,1)} = 2 \mycpr{1} \mycpa{1} + 2 \mycpr{1} + 2 \mycpa{1} + 1 \\
\Delta cp^{(3)} &= 2 \mycpr{2} - 2 \mycpa{2} - \mycpr{1} - \mycpa{1} \\
\Delta cp^{(2,1)} &= \mycpr{2} \mycpa{1} + \mycpr{1} \mycpa{2} + \mycpr{2} + \mycpa{2} + \mycpr{1,1} - \mycpa{1,1} \\
\Delta cp^{(1,1,1)} &= 3 \mycpr{1,1} \mycpa{1} + 3 \mycpr{1} \mycpa{1,1} + 3 \mycpr{1,1} + 6 \mycpr{1} \mycpa{1} + 3 \mycpa{1,1} + 3 \mycpr{1} + 3 \mycpa{1} + 1 
\end{align*}

\subsection{Inductive counting of skew tableaux}

It is now natural to try to count skew tableaux of different forms using structural induction. For integers $n,k \geq 0$, we note the \emph{falling factorial} $(n)_k = n(n-1) \cdots (n-k+1)$, and the number of $k$-tuples in $n$ elements is exactly $(n)_k$. Given a standard tableau of shape $\lambda$ and a small partition $\mu$, we now try to count inductively the number of tuples that leads to a skew tableau of shape $\lambda / \mu$ using the bijection in Lemma~\ref{lem:skew-bijection}. 

We now describe a general scheme for computing such quantities. For a standard tableau $T$, we will see that the number of corresponding skew tableaux can be expressed as a sum through all cells $a \in T$, over some content evaluation on a certain component of $T^a$. We want to compute such quantity inductively for all subtrees in the trace forest of $T$. For such a subtree $S$ rooted at $f$, instead of computing directly the sum we want, we try to find out the inductive form of that sum. The idea is that the sum over $a \in S$ comes in three cases: $a=f$, $a \in S_<$, $a \in S_\vee$. The first case is readily expressed as content evaluation of $S_<$ and $S_\vee$, and the latter two cases consist of a sum of the same type we are computing. They can, hopefully, also be reduced to some content evaluation for $S_<$ and $S_\vee$. We thus obtain the inductive form, and by comparing those of $cp^{\alpha}$, we can identify the sum as a linear combination of content evaluation of $S$.

Before proceeding to examples of application of our scheme, we first deal with some definitions and facts we need. The \emph{conjugate} of a partition $\lambda$, noted as $\lambda^{\dagger}$, is the partition whose Ferrers diagram is that of $\lambda$ flipped alongside the line $y=x$.

\begin{lem} \label{lem:induction-case}
For a skew tableau $T$, a subtree $S$ in its trace forest rooted in $f$, and $a \in S$, we have 3 cases.
\begin{itemize}
\item $a=f$. In this case, $C_<(a,S)=S_<$, $C_\vee(a,S)=S_\vee$.
\item $a \in S_<$. In this case, $C_<(a,S)=C_<(a,S_<) \cup \{ f_< \}$, $C_\vee(a,S)=C_\vee(a,S_<) \cup S_\vee$.
\item $a \in S_\vee$. In this case, $C_<(a,S)=C_<(a,S_\vee) \cup S_<$, $C_\vee(a,S)=C_\vee(a,S_\vee) \cup \{ f_\vee \}$.
\end{itemize}
\end{lem}
\begin{proof}
It follows from Lemma~\ref{lem:trace-forest-separation}.
\end{proof}

Since we will evaluate functions in $\Lambda_c$ on disjoint union of sets, we need the following lemma to ``decompose'' the evaluation.

\begin{prop}\label{prop:union-eval}
For a partition $\alpha = (\alpha_1, \ldots, \alpha_l)$, two disjoint subsets $A, B$ of cells in a tableau $T$ and $a$ an arbitrary cell in $T$, we have
\[ cp^{\alpha}_a(A \cup B) = \sum_{\alpha^{(1)} \uplus \alpha^{(2)} = \alpha} \left( \prod_{i \geq 1} \binom{m(\alpha, i)}{m(\alpha^{(1)},i)} \right) cp^{\alpha^{(1)}}_a(A) cp^{\alpha^{(2)}}_a(B).  \]
Here $\uplus$ means the union of multisets, and $m(\alpha, i)$ (resp. $m(\alpha^{(1)},1)$) is the multiplicity of $i$ in $\alpha$ (resp. $\alpha^{(1)}$).
\end{prop}
\begin{proof}
It follows from the definition of $cp^{(k)}_a$ that
\begin{align*}
cp^{\alpha}_a(A \cup B) &= \prod_{i=1}^{l} \left( cp^{(\alpha_i)}_a(A) + cp^{(\alpha_i)}_a(B) \right) \\
&= \sum_{\alpha^{(1)} \uplus \alpha^{(2)} = \alpha} \left( \prod_{i \geq 1} \binom{m(\alpha, i)}{m(\alpha^{(1)},i)} \right) cp^{\alpha^{(1)}}_a(A) cp^{\alpha^{(2)}}_a(B)
\end{align*}
\end{proof}

We will now investigate some relations on partitions that can simplify some calculations.

\begin{prop} \label{lem:young-lattice}
For a partition $\mu$, let $P(\mu)$ be the set of partitions whose Ferrers diagram can be obtained by adding a cell to that of $\mu$. For any partition $\lambda$, $f^{\lambda / \mu} = \sum_{\mu' \in P(\mu)} f^{\lambda / \mu'}$.
\end{prop}
\begin{proof}
It follows from the classification of all skew tableaux of shape $\lambda / \mu$ by the cell containing $1$.
\end{proof}

\begin{lem} \label{lem:conjugate-mu}
For a partition $\mu$ and its conjugate $\mu^{\dagger}$, if there exists a multivariate function $F$ such that for any $\lambda$ we have $f^{\lambda / \mu} = F(cp^{(1)}(\lambda), cp^{(2)}(\lambda), \ldots, cp^{(i)}(\lambda), \ldots)$, then $f^{\lambda / \mu^{\dagger}} = F(cp^{(1)}(\lambda), -cp^{(2)}(\lambda), \ldots, (-1)^{i-1}cp^{(i)}(\lambda), \ldots)$.
\end{lem}
\begin{proof}
By flipping skew tableaux alongside $y=x$, we see that $f^{\lambda / \mu} = f^{\lambda^{\dagger} / \mu^{\dagger}}$. We then conclude the proof by observing that $cp^{(k)}(\lambda) = (-1)^{k-1} cp^{(k)}(\lambda^{\dagger})$.
\end{proof}

With these simple facts, we now proceed to examples of computing $f^{\lambda / \mu}$ for a fixed $\mu$ with our scheme. We recall that, for a cell $r$ in a trace forest, we denote by $r_<$ the cell to the right of $r$, and by $r_\vee$ the cell above $r$.

\begin{prop} \label{prop:tableau-2}
For a partition $\lambda \vdash n$,
\[ (n)_2 f^{\lambda / (2)} / f^{\lambda} = \frac{1}{2} cp^{(1,1)}(\lambda) + cp^{(2)}(\lambda) - \frac{1}{2} cp^{(1)}(\lambda) = n(n-1)/2 + \sum_{w \in \lambda} c(w). \]
\end{prop}
\begin{proof}
For a subtree $S$ of a trace forest rooted at $r$, we define the following function $G_{(2)}(S) = \sum_{a \in S} cp^{(1)}_{r_<}(C_<(a,S))$. For a standard tableau $T$ and its only tree $F_T$ in its trace forest, by Lemma~\ref{lem:trace-forest-separation}, $G_{(2)}(F_T)$ is the number of tuples $(a,b)$ such that $(T,a,b)$ leads to a skew tableau of shape $\lambda / (2)$ as in Lemma~\ref{lem:skew-bijection}.

We notice that 
\[G_{(2)}(S) = \sum_{a \in S} (\Gamma_+ cp^{(1)}_{r})(C_<(a,S)) = \sum_{a \in S} cp^{(1)}_{r}(C_<(a,S)). \]

We now compute the inductive form of $G_{(2)}$ using Lemma~\ref{lem:induction-case} and Proposition~\ref{prop:union-eval}.

\begin{align*}
(\Delta G_{(2)})(S) &= (cp^{(1)}_{r}(C_<(r,S)) + \sum_{a \in S_<} \left( cp^{(1)}_{r}(C_<(a,S)) - (\Gamma_+ cp^{(1)}_{r})(C_<(a,S_<)) \right) \\
&\quad + \sum_{a \in S_\vee} \left( cp^{(1)}_{r}(C_<(a,S)) - (\Gamma_- \Gamma_+ cp^{(1)}_{r})(C_<(a,S_\vee)) \right) \\
&= \mycpr{1} + \sum_{a \in S_<} 1 + \sum_{a \in S_\vee} \mycpr{1} \\
&= 2\mycpr{1} + \mycpr{1}\mycpa{1} = \left( \Delta \left( \frac{1}{2} cp^{(1,1)} + cp^{(2)} - \frac{1}{2} cp^{1} \right) \right)(S)
\end{align*}

We then have $G_{2} = \frac{1}{2} cp^{(1,1)} + cp^{(2)} - \frac{1}{2} cp^{1}$ by Lemma~\ref{lem:inductive-identify}, and $G_{2}(F_T)$ is thus independent of $T$. By summing over all $T$, we conclude the proof. 
\end{proof}

For $f^{\lambda / (1,1)}$, the formula can be found either with the same approach, or with Lemma~\ref{lem:young-lattice} applied on $\mu=(1)$, or with Lemma~\ref{lem:conjugate-mu}. This proposition entails Theorem~\ref{thm:mu-2}, but without explicitly using any symmetry. We have the first evidence that our scheme may work in more general cases.

We now investigate the next case $\mu = (3)$.

\begin{prop} \label{prop:tableau-3}
For a partition $\lambda \vdash n$,
\[
(n)_{3} f^{\lambda / (3)} / f^{\lambda} = \frac{1}{6} cp^{(1,1,1)}(\lambda) + cp^{(2,1)}(\lambda) + cp^{(3)}(\lambda) - cp^{(1,1)}(\lambda) - 2cp^{(2)} + \frac{5}{6} cp^{(1)}(\lambda)
\]
\end{prop}
\begin{proof}
For a subtree $S$ of a trace forest rooted at $r$, we define a function $G_{(3)}(S) = \sum_{a \in S} G_{(2)}(S^{a}_<)$. For a standard tableau $T$ and its only tree $F_T$ in its trace forest, $G_{(3)}(F_T)$ is the number of tuples $(a,b,c)$ such that $(T,a,b,c)$ leads to a skew tableau of shape $\lambda / (3)$ as in Lemma~\ref{lem:skew-bijection}. We now compute the inductive form of $G_{(3)}$ using Lemma~\ref{lem:induction-case} and Proposition~\ref{prop:tableau-2}.

We notice that
\begin{align*}
G_{(2)}(S) &= \sum_{a \in S} \left(\Gamma_+ \left( \frac{1}{2} cp^{(1,1)} + cp^{(2)} - \frac{1}{2} cp^{1} \right) \right) (C_<(a,S)) \\
&= \sum_{a \in S} \left( \frac{1}{2} cp^{(1,1)} + cp^{(2)} - \frac{3}{2} cp^{1} \right) (C_<(a,S)).
\end{align*}

We now compute the inductive form of $G_{(3)}$ using Lemma~\ref{lem:induction-case} and Proposition~\ref{prop:union-eval}.

\begin{align*}
(\Delta G_{(3)})(S) &= \frac{1}{2} \mycpr{1,1} + \mycpr{2} - \frac{3}{2} \mycpr{1} + 2 \sum_{a \in S_<} cp^{(1)}_{r}(C_<(a,S_<)) \\
&\quad + (\mycpr{1}-1) \sum_{a \in S_\vee} cp^{(1)}_{r}(C_<(a,S_\vee)) + \sum_{a \in S_\vee} \left( \frac{1}{2} \mycpr{1,1} + \mycpr{2} - \frac{1}{2} \mycpr{1} \right) \\
&= \frac{1}{2} \mycpr{1,1} \mycpa{1} + \frac{1}{2} \mycpr{1} \mycpa{1,1} + \mycpr{2} \mycpa{1} + \mycpr{1} \mycpa{2} - \mycpr{1} \mycpa{1} + 3 \mycpr{2} - \mycpa{2} \\
&\quad + \frac{3}{2} \mycpr{1,1} - \frac{1}{2} \mycpa{1,1} - \frac{9}{2} \mycpr{1} - \frac{1}{2} \mycpa{1} \\
&= \left( \Delta \left( \frac{1}{6} cp^{(1,1,1)} + cp^{(2,1)} + cp^{(3)} - cp^{(1,1)} - 2 cp^{(2)} + \frac{5}{6} cp^{(1)} \right) \right)(S)
\end{align*}

We conclude the proof by Lemma~\ref{lem:inductive-identify} as in Proposition~\ref{prop:tableau-2}.
\end{proof}

Combining this proposition with Lemma~\ref{lem:young-lattice} and Proposition~\ref{prop:tableau-2}, we can also compute $f^{\lambda / (2,1)}$ and $f^{\lambda / (1,1,1)}$, and we obtain the character evaluated on a $3$-cycle for $\lambda \vdash n$: 
\[
(n)_3 \chi^{\lambda}_{(3,1^{n-3})} / f^{\lambda} = 3cp^{(3)}(\lambda) - \frac{3}{2} cp^{(1,1)}(\lambda) + \frac{3}{2} cp^{(1)}(\lambda) = 3\sum_{w \in \lambda} (c(w))^2 - 3\binom{n}{2}.
\]
And we have another example of bijective proof of character evaluation formula given by our scheme.

Always following our scheme, with some more tedious but \emph{automated} computation, we obtain the following result for $f^{\lambda / (4)}$. 

\begin{prop} \label{prop:tableau-4}
For a partition $\lambda \vdash n$,
\begin{align*}
(n)_{4} f^{\lambda / (4)} / f^{\lambda} &= \frac{1}{24} cp^{(1,1,1,1)}(\lambda) + \frac{1}{2} cp^{(2,1,1)}(\lambda) + \frac{1}{2} cp^{(2,2)}(\lambda) + cp^{(3,1)}(\lambda) + cp^{(4)}(\lambda) \\
&\quad - \frac{3}{4} cp^{(1,1,1)}(\lambda) - \frac{9}{2} cp^{(2,1)}(\lambda) - \frac{9}{2} cp^{(3)}(\lambda) + \frac{71}{24} cp^{(1,1)}(\lambda) + 6 cp^{(2)}(\lambda) - \frac{9}{4} cp^{(1)}(\lambda)
\end{align*}
\end{prop}

Here we omit the proof, which is essentially a long (but automatic) computation of inductive form. Using Lemma~\ref{lem:conjugate-mu} we also have the expression of $f^{\lambda / (1,1,1,1)}$, and by Lemma~\ref{lem:young-lattice} applied to $\mu = (3)$ and $\mu = (1,1,1)$, we obtain the expression of $f^{\lambda / (3,1)}$ and $f^{\lambda / (2,1,1)}$. We can thus compute $\chi^{\lambda}_{(4,1^{n-4})}$ using Lemma~\ref{lem:character-as-skew-tableau} and we have the following formula:
\[
(n)_4 \chi^{\lambda}_{(4,1^{n-4})} / f^{\lambda} = 4 \sum_{w \in \lambda} (c(w))^3 + 4(2n-3) \sum_{w \in \lambda} c(w)
\]

This is indeed a bijective proof of the character evaluation formula we want. Furthermore, since we can also compute $f^{\lambda / (2,2)}$ using Lemma~\ref{lem:young-lattice}, we can also obtain a bijective proof for the character $\chi^{\lambda}_{(2,2,1^{n-4})}$. As a remark, we notice that our proofs above never depend on the precise structure of the trace forest $F_T$, but rather on the fact that it is a binary tree.

Even though the calculation above seems to be tedious, but it can be totally automatized using all the previous computational lemmas and propositions.

\section{Discussion}

In this article, using the notion of ``trace forest'' which reflects a fine structure in the famous jeu de taquin, we give a simple bijective proof of Theorem~\ref{thm:mu-2} through counting skew tableaux of different shapes. Inspired by this simple proof, we sketch a scheme for counting skew tableaux of more general shapes in an elementary way, using structural induction on the trace forest, and this scheme also leads to combinatorial proofs of several more sophisticated character evaluation formulae. It is also interesting that our proofs can actually be refined to hold for each standard tableau, which still lacks a good explanation.

Empirically, our scheme seems to work for $f^{\lambda / \mu}$ for $\mu$ a hook. And we have the following conjecture. 

\begin{conj}
Our scheme always gives the formula of $f^{\lambda / \mu}$ in terms of contents when $\mu$ is a hook.
\end{conj}

To prove this conjecture, we have to explain two ``miracles'' that occur in computations following to our scheme. First, in the computation of inductive form, we have some kind of sums over $a \in F_<$ and $a \in F_\vee$, but there is no guarantee that these sums can be expressed in $\mycpr{k}$ and $\mycpa{k}$. Second, when we obtain the inductive form in $\mycpr{k}$ and $\mycpa{k}$, we always find that it is the inductive form of some linear combination of $cp^{\alpha}$. This can be seen as a direct consequence of the fact that all character evaluation can be expressed in $cp^{\alpha}$, proved in \cite{corteel2004content} using algebraic methods, but no combinatorial proof is known. 

Unfortunately, for general $\mu$ our scheme does not always work. For instance, our scheme fail to compute $f^{\lambda / (2,2)}$ directly. However we can compute $f^{\lambda / (2,2)} + f^{\lambda / (3,1)}$ and $f^{\lambda / (3,1)}$ with our scheme, and it leads to an expression of $f^{\lambda / (2,2)}$. We have the intuition that our scheme, with Lemma~\ref{lem:young-lattice} and Lemma~\ref{lem:conjugate-mu}, would give enough linear combinations to work out all $f^{\lambda / \mu}$. We prove that this approach works for any $\mu \vdash k \leq 4$, and it extends easily to $\mu \vdash k \leq 6$. More general cases need further investigation.

When passing from $f^{\lambda / \mu}$ to $\chi^{\lambda}_{\mu, 1^{n-k}}$, we notice that $\chi^{\lambda}_{\mu, 1^{n-k}}$ often have a much simpler form, due to some cancellations in the sum. Thus it might be easier to directly deal with the inductive form of characters, and we might see the combinatorial reason behind.

\bibliographystyle{alpha}
\bibliography{jdt-fang}

\end{document}